\documentclass[11pt,reqno,tbtags]{amsart}
\usepackage{amssymb}
\usepackage{natbib}
\bibpunct[, ]{[}{]}{;}{n}{,}{,}

\numberwithin{equation}{section}
\allowdisplaybreaks

\newtheorem{theorem}{Theorem}

\theoremstyle{definition}
\newtheorem{example}[theorem]{Example}
\newtheorem*{definition}{Definition}

\newtheorem{remark}[theorem]{Remark}

\theoremstyle{remark}

\newenvironment{romenumerate}{\begin{enumerate}
 }{\end{enumerate}}

\newcounter{oldenumi}
{\setcounter{oldenumi}{\value{enumi}}
\begin{romenumerate} \setcounter{enumi}{\value{oldenumi}}}
{\end{romenumerate}}

\newcounter{thmenumerate}

\newcounter{xenumerate}   

\newcommand{\refT}[1]{Theorem~\ref{#1}}

\newcommand{\refR}[1]{Remark~\ref{#1}}

\newcommand{\refand}[2]{\ref{#1} and~\ref{#2}}

\newcommand\marginal[1]{\marginpar{\raggedright\parindent=0pt\tiny #1}}

\begingroup
  \divide\count255 by 60
  \multiply\count255 by -60
  \advance\count255 by \time
  \ifnum \count255 < 10 \xdef\klockan{\the\count1.0\the\count255}
  \else\xdef\klockan{\the\count1.\the\count255}\fi
\endgroup

\newcommand\set[1]{\ensuremath{\{#1\}}}
\newcommand\bigset[1]{\ensuremath{\bigl\{#1\bigr\}}}

\def\rompar(#1){\textup(#1\textup)}    

\def\xexp(#1){e^{#1}}

\newcommand\iid{i.i.d.\spacefactor=1000}    
\newcommand\ie{i.e.\spacefactor=1000}

\newcommand\eqd{\overset{\mathrm{d}}{=}}

\newcommand\bbR{\mathbb R}

\newcommand\bbN{\mathbb N}

\newcommand\bbQ{\mathbb Q}

\newcounter{CC}
\newcounter{cc}

\renewcommand\P{\operatorname{\mathbb P{}}}

\newcommand\ga{\alpha}

\newcommand\gf{\varphi}

\newcommand\go{\omega}
\newcommand\gO{\Omega}
\newcommand\gs{\sigma}

\newcommand\cA{\mathcal A}
\newcommand\cB{\mathcal B}
\newcommand\cC{\mathcal C}
\newcommand\cD{\mathcal D}

\newcommand\cF{\mathcal F}

\newcommand\cS{{\mathcal S}}

\newcommand\tf{{\tilde f}}
\newcommand\tg{{\tilde g}}
\newcommand\tX{{\tilde X}}

\newcommand\ett[1]{\boldsymbol1[#1]} 

\def\[#1]{[\![#1]\!]}

\newcommand\qw{^{-1}}

\renewcommand{\=}{:=}

\newcommand\oi{[0,1]}

\newcommand\inn{_1^n}

\newcommand\OFP{\ensuremath{(\Omega,\cF,P)}}
\newcommand\OIL{\ensuremath{(\oi,\cB,dx)}}
\newcommand\bbNx{\bbN^*}
\newcommand\tOFP{\ensuremath{(\tgO,\tcF,\tP)}}
\newcommand\tgO{\tilde\gO}
\newcommand\tcF{\tilde\cF}
\newcommand\tP{\tilde P}
\newcommand\soo{\mathfrak S_\infty}
\newcommand\mi{{m_i}}
\newcommand\tiota{\tilde\iota}

\newcommand{\Lovasz}{Lov\'asz}

\newcommand\REM[1]{{\raggedright\texttt{[#1]}\par\marginal{XXX}}}

\newcommand\urladdrx[1]{{\urladdr{\def~{{\tiny$\sim$}}#1}}}

\begin{document}
\title[Representation of multivariate functions]
{Standard representation of multivariate functions 
 on a general probability space}

\date{December 21, 2007} 

\author{Svante Janson}
\address{Department of Mathematics, Uppsala University, PO Box 480,
SE-751~06 Uppsala, Sweden}
\email{svante.janson@math.uu.se}
\urladdrx{http://www.math.uu.se/~svante/}

\subjclass[2000]{} 

\begin{abstract} 
It is well-known that a random variable, \ie{} a function defined on a
probability space, with values in a Borel space, can be represented on
the special probability space consisting of the unit interval with
Lebesgue measure.
We show 
an extension of this to multivariate
functions.
This is motivated by some recent constructions of random graphs.
\end{abstract}

\maketitle


A fundamental feature of probability theory is that random variables
are defined as functions on a probability space \OFP, but the actual
probability space has very little importance, and often no importance
at all. For example, we have the following well-known basic result,
showing that 
as long as we consider a single random variable, we may assume that
the probability space is \OIL, the unit interval with the Borel
$\gs$-field and Lebesgue measure. 
\begin{theorem}\label{T0}
  If\/ $f:\OFP\to\bbR$ is any random variable, then there exists a
  random variable $\tf:\OIL\to\bbR$ such that $f$ and $\tf$ have the
  same distribution.
\end{theorem}
A standard construction of $\tf$ is to take the right-continuous
inverse of the distribution function $\P(f\le x)$, 
$\tf(x)\=\inf\set{y:\P(f\le y)>x}$. 

Moreover, it is also well-known that
\refT{T0} extends to vector-valued random variables and random
variables with values in more general spaces.
To state a precise result, we recall that a \emph{Borel space}
\cite[Appendix A1]{Kallenberg}, 
also called Lusin space \cite[III.16, III.20(b)]{DM}, is a measurable 
space that is isomorphic to a Borel subset of \oi.
Every Polish space (a complete separable metric space) with
its Borel $\gs$-field is a Borel space
\cite[Appendix A1]{Kallenberg}, 
\cite[III.17]{DM}.
(For example, this includes $\bbR^n$, which gives the extension to
vector valued random variables.)

\begin{theorem}\label{T0S}
  If\/ $f:\OFP\to S$ is any random variable with values in a Borel
  space $S$, then there exists a
  random variable $\tf:\OIL\to S$ such that $f$ and $\tf$ have the
  same distribution.
\end{theorem}

\begin{proof}
This is an almost trivial 
extension of \refT{T0}: By assumption, there exists an
isomorphism $\gf:S\to E$, where $E\subseteq\oi$ is a Borel set. This
means that $\gf$ is a bijection, and that $\gf$ and $\gf\qw:E\to S$ are
measurable. Extend $\gf\qw$ to $\bbR$ by defining $\gf\qw(x)=s_0$
when $x\notin E$ for some arbitrary $s_o\in S$.
Then, $g\=\gf\circ f:\OFP\to E\subset\bbR$ is a random
variable, and \refT{T0} gives $\tg:\OIL\to\bbR$ with 
$\tg\eqd g$. Let $\tf\=\gf\qw\circ\tg:\OIL\to S$.
\end{proof}

\begin{remark}
  \label{Rcont}
In fact, every uncountable Borel space $T$ is isomorphic to \oi{}
\cite[Appendix III.80]{DM}. 
If $\mu$ is any continuous
probability measure 
(\ie, a probability measure such that every point has measure 0) 
on a Borel space $T$, then $T$ necessarily is
uncountable, and thus there is an isomorphism $\ga:T\to\oi$ which maps
$\mu$ to some continuous Borel measure $\nu$ on \oi, and hence $\ga$
is an
isomorphism $(T,\mu)\to(\oi,\nu)$. 
Further, if we let $F_\nu(x)\=\nu[0,x]$, then $F_\nu:\oi\to\oi$ is a
continuous function that maps $\nu$ to the Lebesgue measure
$dx$. Hence, in \refT{T0}, $\tf\circ F_\nu\circ\ga:(T,\mu)\to\bbR$ has 
$\tf\circ F_\nu\circ\ga\eqd\tf\eqd f$, and similarly for \refT{T0S}.
Consequently, we can replace $\OIL$ by $(T,\mu)$ in Theorems
\refand{T0}{T0S}, for any Borel space $T$ with a continuous
probability measure $\mu$.
\end{remark}

The purpose of this note is to give an 
elementary proof of an extension of \refT{T0} to 
functions $f:\gO^m\to\bbR$ or $f:\gO^m\to S$
of several variables, where $m\ge1$.
This is a folklore result, but we do not know any reference, so we
provide a detailed statement and proof.
Our interest in this problem comes from the following example, with a
construction of random graphs used by, for example,
\citet{LSz} and \citet{SJ178} (in different contexts), 
see further \citet{BCL1,BCL1ii,SJ209} 
and further references in these papers,
and
an extension of this example studied in \citet{clustering}, where also
functions $f:\gO^m\to\oi$ with 
$m>2$ are used.
We use the notation $[n]\=\set{1,\dots,n}$ if
  $n<\infty$ and
 $[\infty]\=\bbN\=\set{1,2,\dots}$.

\begin{example}
  Let $f:\gO^2\to\oi$ be a symmetric measurable function and let $1\le
  n\le\infty$.
Define a random graph on the vertex set 
$[n]$
by first taking
$n$ \iid{} random variables $(X_i)\inn$ in $\gO$ with
distribution $P$, and then letting, conditioned on these random variables,
the edges $ij$ with $i<j$ appear independently, with the
probability of an edge $ij$ equal to $f(X_i,X_j)$.
\end{example}

If we choose another probability space $\tOFP$ and a
symmetric function $\tf:\tgO^2\to\oi$, we obtain the same
distribution of the corresponding two random graphs if the joint
distribution of the families $(f(X_i,X_j))_{i<j}$ and
$(\tf(\tX_i,\tX_j))_{i<j}$ are equal, where $\tX_i$ has the
distribution $\tP$ on $\tgO$.
This motivates the following definition, where $\soo$ denotes the set
of all permutations of $\bbN$.
(For notational convenience we write $f(X_1,\dots,X_m)$ also when
$m=\infty$; this should be interpreted as $f(X_1,X_2,\dots)$.)

\begin{definition}
  Let \OFP{} be a probability space and let 
$f_\ga:\gO^{m_\ga}\to S_\ga$, $\ga\in A$, be a collection of
  measurable functions, where 
$A$ is any index set,
$m_\ga\in\bbNx\=\bbN\cup\set\infty$ 
are positive integers or
  $\infty$, and $S_\ga$ are measurable  spaces.
We say that the family $(f_\ga)_\ga$ can be \emph{represented on a
  probability space $\tOFP$} if there exists 
a collection of measurable functions
$\tf_\ga:\tgO^{m_\ga}\to S_\ga$, $\ga\in A$, such that
if $X_1,X_2,\dots$ are \iid{} random elements of $\gO$ with distribution $P$,
  and 
if $\tX_1,\tX_2,\dots$ are \iid{} random elements of $\tgO$ with distribution
  $\tP$,
then the collections
$\bigl\{f_\ga(X_{\gs(1)},\dots,X_{\gs(m_\ga)}): \ga\in  A,\,\gs\in\soo\bigr\}$
and
$\bigl\{\tf_\ga(\tX_{\gs(1)},\dots,\tX_{\gs(m_\ga)}): 
 \ga\in A,\,\gs\in\soo\bigr\}$
have the same distribution.
\end{definition}
Equivalently, if we regard each $f_\ga$ as being defined on
$\gO^\infty$ by ignoring all but the $m_\ga$ first coordinates, and
let $\soo$ act on $\gO^\infty$ in the natural way, then the
collections
$\bigset{f_\ga\circ\gs:\ga\in A,\,\gs\in\soo}$ and 
$\bigset{\tf_\ga\circ\gs:\ga\in A,\,\gs\in\soo}$
of random variables defined of $\gO^\infty$ and $\tgO^\infty$, respectively,
have the same distribution. 

If we have only a single function $f:\gO^m\to\cS$, 
it is further equivalent that the collections
$f(X_{i_1},\dots,X_{i_m})$
and $\tf(\tX_{i_1},\dots,\tX_{i_m})$, indexed by sequences
$(i_1,\dots,i_m)$ of distinct integers in $\bbN$, have the same distribution.
In principle, the definition could be stated in this way for  families
\set{f_\ga} and \set{\tf_\ga} too, with in general varying (and
possibly unbounded) $m_\ga$, but we leave it to the reader to find a nice
formulation in this form.

\begin{remark}\label{Rcomplete}
We may equip $\gO^{m_\ga}$ either with the product $\gs$-field
$\cF^{m_\ga}$ or with its completion. This does not matter, because if
$f_\ga$ is measurable with respect to the completed $\gs$-field, it
can be modified on a set of (product) measure 0 such that it becomes
measurable with respect to $\cF^{m_\ga}$, and this does not change any
of the distributions above. Similarly, we may assume that $\tf_\ga$ is
measurable with respect to $\tcF^{m_\ga}$.
\end{remark}

\begin{remark}
  As is well-known, the representing functions $\tf_\ga$ are in
  general far from unique.
See, for example,
\cite[Theorem 7.28]{Kallenberg:exch} for a general result on
  equivalence of representations.
\end{remark}

The main result is that every countable collection of functions on any
probability space, with values in a Borel space, can be represented on $\OIL$.

\begin{theorem}\label{T1}
Let $f_i:\gO^{m_i}\to S_i$, $i=1,2,\dots$, be a finite or countable family of
  measurable functions, where $m_i\in\bbNx$ and $S_i$ are Borel spaces.
Then $(f_i)_i$ can be represented on $\OIL$.

Moreover, if every $f_i$ is symmetric, then we can choose the
  representing functions on $\OIL$ to be symmetric too.
\end{theorem}

\begin{proof}
  First, by the argument in the proof of \refT{T0S}, it suffices to
  consider the 
  case $S_i=\bbR$ for all $i$. Further, we may by \refR{Rcomplete}
  assume that each $f_i$ is measurable with respect to $\cF^{m_\ga}$,
  without completion.

Let $\cA\=\set{\cC\subseteq\cF:|\cC|\le\aleph_0}$ 
be the collection of all countable families of measurable sets in
$\gO$.
Since a countable union $\bigcup_j \cC_j$ of families $\cC_j\in\cA$
also belongs to $\cA$, it is easily seen that
$\bigcup_{\cC\in\cA}\gs(\cC)^m$ is a $\gs$-field on $\gO^m$, for every
$m\le\infty$; since
further, by the definition of product $\gs$-fields, every set in
$\cF^m$ belongs to a sub-$\gs$-field generated by a countable number
of cylinder sets, and thus to some $\gs(\cC)^m$, it follows that
$\cF^m=\bigcup_{\cC\in\cA}\gs(\cC)^m$.
Moreover, every $f_i$ is measurable with respect to the
$\gs$-field generated by the countable family 
$\set{\go'\in\gO^{m_\ga}:f_i(\go')<r}$, $r\in\bbQ$, and 
each of these thus
belongs to the product $\gs$-field $\cF_{i,r}^{m_\ga}$ defined by some
countably generated sub-$\gs$-field $\cF_{i,r}$ of $\cF$. 
Consequently, there exists a countably generated sub-$\gs$-field $\cF_0$
of $\cF$ such that each $f_i$ is $\cF_0^{m_i}$ measurable.

Next, let $A_1,A_2,\dots$ be a sequence of subsets of $\gO$ that
generate $\cF_0$. Let $h:\gO\to \cD\=\set{0,1}^\infty$ be defined by
$h(x)=(\ett{x\in A_i})_{i}$. 
($\cD$ is, topologically, the Cantor set.)
Then $\cF_0$ equals the $\gs$-field
generated on $\gO$ by $h$, and thus $\cF_0^m$ equals the $\gs$-field
on $\gO^m$ generated by $h^m\=(h,\dots,h):\gO^m\to \cD^m$.
Since $f_i$ is $\cF_0^\mi$-measurable, it follows that $f_i=g_i\circ h^\mi$
for some measurable function $g_i:\cD^\mi\to\bbR$. If we let $\tgO=\cD$,
let $\tcF$ be the Borel $\gs$-field on $\cD$, and let $\tP$ be the
probability measure induced by $h$, it follows immediately that
the family $(g_i)_i$ gives a representation of $(f_i)_i$ on $(\cD,\tP)$.

Applying \refT{T0S} to the identity function
$\iota:(\cD,\tP)\to(\cD,\tP)$, we see that there exists a function
$\tiota:\oi\to\cD$ mapping the Lebesgue measure $dx$ to $\tP$.
Consequently, the mappings $g_i\circ \tiota^{m_i}\to\bbR$ represent
$(f_i)_i$ on \OIL.

It is clear that $g_i$ is symmetric if every $f_i$ is.
\end{proof}

The symmetry statement in \refT{T1} can be extended to include partial
symmetries, antisymmetries, \dots; we leave this to the reader.

\begin{remark}
\refT{T1} is related to the theory for
exchangeable arrays by \citet{Aldous} and \citet{Hoover}, see
\citet[Chapter 7]{Kallenberg:exch} for a detailed exposition.
Consider for simplicity a single function $f:\gO^2\to\cS$.
It is obvious that the random variables
$f(X_{i},X_{j})$, $i\neq j$, form an exchangeable family.
Conversely, the theorem by Aldous--Hoover 
\cite[Theorem 7.22]{Kallenberg:exch}
says that every exchangeable family can be represented on $\OIL$, but
in the more general form $g(X_\emptyset,X_i,X_j,X_{ij})$. It seems
likely that a further study of the representations can lead to a proof
that we may in this case take $g$ independent of $X_\emptyset$ and
$X_{ij}$, which would yield \refT{T1} in this case. (And presumably
this could be extended to the general case.)
Nevertheless, such a proof would be highly technical, and hardly
shorter than the elementary proof given above, so we have not pursued this.  
\end{remark}

\newcommand\AAP{\emph{Adv. Appl. Probab.} }
\newcommand\JAP{\emph{J. Appl. Probab.} }
\newcommand\JAMS{\emph{J. \AMS} }
\newcommand\MAMS{\emph{Memoirs \AMS} }
\newcommand\PAMS{\emph{Proc. \AMS} }
\newcommand\TAMS{\emph{Trans. \AMS} }
\newcommand\AnnMS{\emph{Ann. Math. Statist.} }
\newcommand\AnnPr{\emph{Ann. Probab.} }
\newcommand\CPC{\emph{Combin. Probab. Comput.} }
\newcommand\JMAA{\emph{J. Math. Anal. Appl.} }
\newcommand\RSA{\emph{Random Struct. Alg.} }
\newcommand\ZW{\emph{Z. Wahrsch. Verw. Gebiete} }
\newcommand\DMTCS{\jour{Discr. Math. Theor. Comput. Sci.} }

\newcommand\AMS{Amer. Math. Soc.}
\newcommand\Springer{Springer-Verlag}
\newcommand\Wiley{Wiley}

\newcommand\vol{\textbf}
\newcommand\jour{\emph}
\newcommand\book{\emph}
\newcommand\inbook{\emph}
\def\no#1#2,{\unskip#2, no. #1,} 
\newcommand\toappear{\unskip, to appear}

\newcommand\webcite[1]{
   \penalty0\texttt{\def~{{\tiny$\sim$}}#1}\hfill\hfill}
\newcommand\webcitesvante{\webcite{http://www.math.uu.se/~svante/papers/}}
\newcommand\arxiv[1]{\webcite{arXiv:#1.}}

\def\nobibitem#1\par{}

\end{document}